\documentclass[11pt,a4paper]{article}
\usepackage[utf8]{inputenc}
\usepackage[english]{babel}
\usepackage{hyperref}
\usepackage{amsmath,amsthm,enumerate}
\usepackage{amssymb}
\usepackage{tikz}\usetikzlibrary{patterns}
\usepackage{subfigure}
\usepackage[hmargin=2.6cm,vmargin=2.8cm]{geometry}
\usepackage[color=green!30]{todonotes}

\newtheorem{theorem}{Theorem}

\newtheorem{proposition}[theorem]{Proposition}

\newtheorem{corollary}[theorem]{Corollary}

\begin{document}

\title{On powers of interval graphs and their orders}
\author{Florent Foucaud\footnote{\noindent LIMOS - CNRS UMR 6158, Universit\'e Blaise Pascal, Clermont-Ferrand (France). florent.foucaud@gmail.com}
\and Reza Naserasr\footnote{\noindent LIAFA - CNRS UMR 7089, Université Paris Diderot - Paris 7 (France). reza@lri.fr}
\and Aline Parreau\footnote{\noindent LIRIS - CNRS UMR 5205, Université Lyon 1 (France). aline.parreau@univ-lyon1.fr}
\and Petru Valicov\footnote{\noindent LIF - CNRS UMR 7279, Université d'Aix-Marseille (France). petru.valicov@lif.univ-mrs.fr}}

\maketitle

\vspace{-5mm}
\noindent\fbox{
  \begin{minipage}{\textwidth}
    {\bf Note to readers: it has come to our attention that Theorem~\ref{thm:main}, the main result of this note, follows from earlier results of Agnarsson, Damaschke and Halld\'orsson~\cite{ADH}.}
  \end{minipage}
}

\begin{abstract}
It was proved by Raychaudhuri in 1987 that if a graph power $G^{k-1}$ is an interval graph, then so is the next power $G^k$. 
This result was extended to $m$-trapezoid graphs by Flotow in 1995. 
We extend the statement for interval graphs by showing that any interval representation of $G^{k-1}$ can be extended to an interval representation of $G^k$ that induces the same left endpoint and right endpoint orders. The same holds for unit interval graphs. We also show that a similar fact does not hold for trapezoid graphs.
\end{abstract}

\section{Introduction}

An interval graph is a graph whose vertices correspond to intervals of the real line, and such that two vertices are adjacent if and only if the corresponding intervals intersect. More generally, given a set $S$ of geometric objects, the \emph{intersection graph} $G$ of $S$ is the graph with vertex set $S$ such that two vertices are adjacent if and only if they have a nonempty intersection. We say that $S$ is a \emph{geometric representation} of $G$. Let $m\geq 0$ be an integer and consider $m+1$ parallel horizontal lines $L_1,\ldots,L_{m+1}$ indexed from bottom to top. An \emph{$m$-trapezoid} on $L_1,\ldots,L_{m+1}$ is determined by a set of $m+1$ intervals $I_1,\ldots,I_{m+1}$ with $I_i\subset L_i$ for each $i\in\{1,\ldots,m+1\}$. This $m$-trapezoid consists of the polygon with corners $\ell(I_1),\ldots,\ell(I_{m+1}),r(I_{m+1}),\ldots,r(I_1),\ell(I_1)$ connected in this order (where $\ell(I_i)$ and $r(I_i)$ denote the left an right endpoint of interval $I_i$, respectively). An \emph{$m$-trapezoid graph}~\cite{F95} is the intersection graph of a set of $m$-trapezoids defined on the same $m+1$ lines. Note that $0$-trapezoid graphs coincide with interval graphs, and $1$-trapezoid graphs are simply called trapezoid graphs (and were introduced in~\cite{DCP88}). An interval representation is \emph{proper} if no interval properly contains any other interval, and it is \emph{unit} if all intervals have unit length. A graph is a unit (proper, respectively) interval graph if it is the intersection graph of a unit (proper, respectively) interval representation. It is well-known that an interval graph is unit if and only if it is proper~\cite{R69}.

An $m$-trapezoid representation of an $m$-trapezoid graph $G$ naturally induces $2(m+1)$ orders for $V(G)$ corresponding to the orders of the left and right endpoints of the intervals of each line $L_i$\footnote{Note that two elements of $V(G)$ can be equal for some orders if they share a common left or right endpoint.}. We denote by $\leqslant_L^i$ the order corresponding to the left endpoints of the intervals on line $L_i$, and by $\leqslant_R^i$ the order corresponding to the right endpoints of the intervals on line $L_i$. If $m=0$ (i.e. we consider an interval representation of an interval graph), we simply denote these two orders $\leqslant_L$ and $\leqslant_R$.

Given a graph $G$ and an integer $k\geq 1$, the $k$th \emph{power} $G^k$ of $G$ is the graph with vertex set $V(G)$ where two vertices are adjacent in $G^k$ if and only if they are at distance at most~$k$ in $G$.

Raychaudhuri~\cite{R87} proved that for any graph $G$ and any integer $k\geq 2$, if $G^{k-1}$ is an interval graph, then so is $G^k$. Raychaudhuri also proved the same statement for unit interval graphs. Flotow~\cite{F95} extended Raychaudhuri's result as follows: for any $k\geq 2$ and every $m\geq 0$, if $G^{k-1}$ is an $m$-trapezoid graph, then $G^k$ is an $m$-trapezoid graph. Flotow also proved that the same statement is true for co-comparability graphs~\cite{F95}, and studied similar question for circular-arc graphs~\cite{F96}. Note that a similar result does not hold for the related class of permutation graphs (which are exactly the graphs that are both comparability and co-comparability graphs), since the path $P_7$ is a permutation graph, but its square $P_7^2$ is not (indeed it is not a comparability graph). Similarly this does not hold for chordal graphs (see~\cite{LS83} and~\cite{R87} for some discussion and related results).

Raychaudhuri's proof for (unit) interval graphs relies on classic characterizations of these graph classes in terms of forbidden structures, and he does not consider the interval representations of the graphs in question. In contrast, in Flotow's proof, an $m$-trapezoid representation of $G^{k-1}$ is extended in an inductive process to obtain an $m$-trapezoid representation of $G^{k}$. However, Flotow's proof does not yield any conclusion about the \emph{orders} induced by the two considered $m$-trapezoid representations of $G^{k-1}$ and $G^k$. We prove the following strengthening of Raychaudhuri's results from~\cite{R87}, providing a shorter proof for them.

\begin{theorem}\label{thm:main}
Let $G$ be a graph and $k\geq 2$ an integer such that $G^{k-1}$ is an interval graph. Given any interval representation $R$ of $G^{k-1}$, $R$ can be extended to an interval representation $R'$ of $G^k$ such that $R$ and $R'$ induce the same left and right endpoint orders.
\end{theorem}

\begin{corollary}\label{cor:unit}
Let $G$ be a graph and $k\geq 2$ an integer such that $G^{k-1}$ is a proper interval graph. Given any proper (respectively unit) interval representation $R$ of $G^{k-1}$, $R$ can be extended to a proper (respespectively unit) interval representation $R'$ of $G^k$ such that $R$ and $R'$ induce the same left and right endpoint orders.
\end{corollary}

Finally, we show that a statement similar to the one of Theorem~\ref{thm:main} is not true for $m$-trapezoid graphs in the case $m=1$ and $k=2$ (i.e. for squares of trapezoid graphs).

\begin{proposition}\label{prop:P5}
There is a trapezoid representation $R$ of the path $P_5$ such that for any trapezoid representation $R'$ of $P_5^2$, at least one order among the four orders $\leqslant_L^0$, $\leqslant_R^0$, $\leqslant_L^1$, $\leqslant_R^1$ induced by each of $R$ and $R'$, differs on $R$ and $R'$.
\end{proposition}

We remark that the proof of Theorem~\ref{thm:main} provides a polynomial-time algorithm for building the representation $R'$ from $R$ and $G$. Hence Theorem~\ref{thm:main} has (algorithmic) applications, see for example~\cite{algopaper}.


\section{Proofs}

\begin{proof}[Proof of Theorem~\ref{thm:main}]
We can assume that $G$ is connected and that in the representation $R$ there is no pair of intervals $I_x$ and $I_y$ with $r(I_x)=\ell(I_y)$ (otherwise we can modify $R$ so that it satisfies this property, without affecting $\leqslant_L$ and $\leqslant_R$).

For every vertex $x\in V(G^{k-1})$, we denote by $I_x$ the interval corresponding to $x$ in the representation $R$ of $G^{k-1}$. 
Assume first that there exists a vertex at distance $k$ of $x$ (in $G$) whose corresponding interval of $R$ has a left endpoint larger than $\ell(I_x)$. Let $u_x$ be such a vertex whose corresponding interval $I_{u_x}$ of $R$ has the largest left endpoint. We define $r_k(x)$ to be a point of the real line located after $\ell(I_{u_x})$ and before the next left or right endpoint in $R$ (if it exists).
In the case $u_x=u_y$ for two distinct vertices $x$ and $y$ such that $x\leqslant_R y$, we choose $r_k(x)$ and $r_k(y)$ such that $r_k(x)<r_k(y)$ if $r(I_x)<r(I_y)$ and $r_k(x)=r_k(y)$ if $r(I_x)=r(I_y)$.
If there is no interval $u_x$, we let $r_k(x)=r(I_x)$. In this case, each vertex $y$ whose corresponding interval of $R$ starts after $\ell(I_x)$ is at distance at most~$k-1$ of $x$, and no interval in $R$ starts after $r(I_x)$. 

Now, we build $R'$ as follows: each interval $I_x=[\ell(I_x),r(I_x)]$ is replaced by interval $I'_x=[\ell(I_x),r_k(x)]$. In other words, all the intervals of $R$ are extended to the right until being adjacent to the last interval at distance at most~$k$ in $G$, while locally preserving the order $\leqslant_R$ in the event of ties.

It is clear that $R$ and $R'$ induce the same order $\leqslant_L$ since we have not modified the left endpoints. Assume for contradiction that $R$ and $R'$ do not induce the same order $\leqslant_R$. Let $x$ and $y$ be two vertices. If $r_k(x)=r_k(y)$ we necessarily have $r(I_x)=r(I_y)$. Thus there exist two vertices $x$ and $y$ with $x\leqslant_R y$ in $R$, but $y\leqslant_R x$ in $R'$. In other words, $r(I_x)\leq r(I_y)$ but $r_k(x)>r_k(y)$. This implies in particular that $r(I_x)<r_k(x)$ and hence $u_x$ is well-defined. 
Moreover, by definition of $r_k(x)$, we cannot have $\ell(I_{u_x})\leq r_k(y)<r_k(x)$, therefore $r_k(y)<\ell(I_{u_x})$ and the distance $d_G(y,u_x)$ is at least $k+1$. Since $d_G(x,u_x)=k$, there is a vertex $z$ at distance~$1$ of $u_x$ in $G$ that is at distance $k-1$ of $x$ in $G$ (for example $z$ lies on a shortest path from $u_x$ to $x$ in $G$). Hence, $I_{z}$ intersects both $I_x$ and $I_{u_x}$ in $R$. But then, $I_y$ also intersects $I_{z}$, implying that $d_G(y,u_x)\leq d_G(y,z)+d_G(z,u_x)\leq k-1+1=k$, a contradiction. Therefore $R$ and $R'$ induce the same order $\leqslant_R$.

It remains to show that the interval graph $G'$ defined by $R'$ is exactly $G^k$, i.e. that (i) all edges of $G^k$ are contained in $G'$, and (ii) that every edge of $G'$ belongs to $G^k$.

Note that each interval $I'_x$ of $R'$ contains the interval $I_x$ of $R$, hence all the edges of $G^{k-1}$ are contained in $G'$. Hence, assuming that (i) is false, we have two vertices $x,y$ at distance exactly~$k$ in $G$ that are not adjacent in $G'$. Assume without loss of generality that $x\leqslant_R y$. But then, we have $r_k(x)>\ell(I_y)$ and therefore $I'_x$ and $I'_y$ do intersect in $R'$, a contradiction. Therefore (i) holds.

Now, assume that (ii) is false; then we have two vertices $x,y$ with $d_G(x,y)\geq k+1$ but $I'_x$ intersects $I'_y$. Without loss of generality assume that $x\leqslant_Ry$. Then we have $r_k(x)>\ell(I_y)$, and $x\neq u_x$. Let $z$ be a vertex at distance~$1$ of $x$ and at distance at most~$k-1$ of $u_x$ (for example $z$ lies on a shortest path from $x$ to $u_x$ in $G$). Then, $I_z$ intersects both $I_x$ and $I_{u_x}$ in $R$, which implies that $I_z$ intersects $I_y$ in $R$. Hence, we have $d_G(y,x)\leq d_G(y,z)+d_G(z,x)\leq k-1+1=k$, a contradiction. Hence (ii) holds and the proof of Theorem~\ref{thm:main} is complete.
\end{proof}

\medskip

\begin{proof}[Proof of Corollary~\ref{cor:unit}]
An interval representation is proper if and only if the two orders $\leqslant_L$ and $\leqslant_R$ are the same.

Let $G$ be a graph such that $G^{k-1}$ is a proper interval graph and let $R$ be a proper interval representation of $G^{k-1}$. By Theorem~\ref{thm:main}, $R$ can be extented to an interval representation $R'$ of $G^k$ inducing the same left and right endpoint orders. Thus $R'$ is necessarily a proper interval representation.

For unit interval representations, it is enough to note that any proper interval representation $R$ of a graph $H$ can be transformed into a unit interval representation $R^u$ of $H$ such that $R$ and $R^u$ induce the same left and right endpoint orders (see for example~\cite{BW99}).\end{proof}

\medskip

\begin{proof}[Proof of Proposition~\ref{prop:P5}]
Let $V(P_5)=\{1,2,3,4,5\}$ and consider the following trapezoid representation, $R$, of $P_5$. Each vertex $i$ ($1\leq i\leq 5$) corresponds to the trapezoid $T_i$ with corners $\ell(I_i^0),\ell(I_i^1),r(I_i^1),r(I_i^0)$, where the intervals on $L_0$ are $I_1^0=[0,1]$, $I_2^0=[6,7]$, $I_3^0=[4,5]$, $I_4^0=[10,11]$, $I_5^0=[8,9]$ and the intervals on $L_1$ are $I_1^1=[4,5]$, $I_2^1=[3,4]$, $I_3^1=[8,9]$, $I_4^1=[6,7]$, $I_5^1=[12,13]$. See Figure~\ref{fig:P5} for an illustration. Note that we have $1\leqslant_L^0 3\leqslant_L^0 2\leqslant_L^0 5\leqslant_L^0 4$, $2\leqslant_L^1 1\leqslant_L^1 4\leqslant_L^1 3\leqslant_L^1 5$, $\leqslant_L^0=\leqslant_R^0$ and $\leqslant_L^1=\leqslant_R^1$.

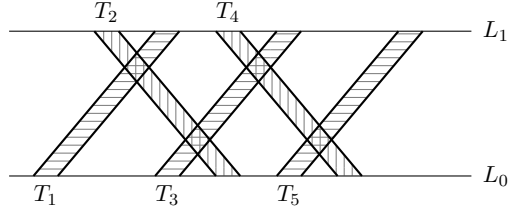
\begin{figure}[ht]
\centering
\scalebox{0.8}{\begin{tikzpicture}[scale=0.4]
\foreach \I in {0,5,10}
         {
           \fill[pattern color=gray, draw=white, pattern=horizontal lines] (\I,-3) -- (\I+5,3) -- (\I+1+5,3) -- (\I+1,-3);
         }

\foreach \I in {2.5,7.5}
         {
           \fill[pattern color=gray, draw=white, pattern=vertical lines] (\I,3) -- (\I+5,-3) -- (\I+1+5,-3) -- (\I+1,3);
         }

\foreach \I in {0,5,10}
         {
           \draw[line width=1pt] (\I,-3) -- (\I+5,3) (\I+1,-3) -- (\I+1+5,3);
         }

\foreach \I in {2.5,7.5}
         {
           \draw[line width=1pt] (\I,3) -- (\I+5,-3) (\I+1,3) -- (\I+1+5,-3); 
         }

\draw[line width=0.5pt] (-1,3) -- (18,3) 
      (-1,-3) -- (18,-3);
\node at (19,3) {$L_1$};
\node at (19,-3) {$L_0$};
\node at (0.5,-3.8) {$T_1$};
\node at (3,3.8) {$T_2$};
\node at (5.5,-3.8) {$T_3$};
\node at (8,3.8) {$T_4$};
\node at (10.5,-3.8) {$T_5$};
\end{tikzpicture}}
\label{fig:P5}
\caption{The trapezoid representation $R$ of graph $P_5$.}
\end{figure}

By contradiction, we assume that there exists a trapezoid representation $R'$ of $P_5^2$ inducing the same orders $\leqslant_L^0$, $\leqslant_R^0$, $\leqslant_L^1$ and $\leqslant_R^1$ as $R$. For each vertex $i$ ($1\leq i\leq 5$), denote by $T'_i$ the trapezoid corresponding to $i$ in $R'$, and by ${I'}_i^0$ and ${I'}_i^1$ the intervals of $L_0$ and $L_1$, respectively, belonging to $T_i$.

Since $3$ and $5$ are adjacent in $P_5^2$, we have $T'_3\cap T'_5\neq\emptyset$. Since $3\leqslant_L^0 5$, $3\leqslant_R^0 5$, $3\leqslant_L^1 5$ and $3\leqslant_R^1 5$, we have $\ell({I'}_5^1)<r({I'}_3^1)$ or $\ell({I'}_5^0)<r({I'}_3^0)$. Since $3\leqslant_R^0 2$ but $2$ and $5$ are not adjacent in $P_5^2$, we have $r({I'}_3^0)<r({I'}_2^0)<\ell({I'}_5^0)$. This implies that $\ell({I'}_5^1)<r({I'}_3^1)$. Since $2$ and $4$ are adjacent in $P_5^2$ and moreover $2\leqslant_L^0 4$, $2\leqslant_R^0 4$, $2\leqslant_L^1 4$ and $2\leqslant_R^1 4$, we must have $\ell({I'}_4^1)<r({I'}_2^1)$ or $\ell({I'}_4^0)<r({I'}_2^0)$. But since $r({I'}_2^0)<\ell({I'}_5^0)<\ell({I'}_4^0)$, we necessarily have $\ell({I'}_4^1)<r({I'}_2^1)$. Now, it suffices to observe that $1$ and $2$ are adjacent but $1$ and $4$ are non-adjacent in $P_5^2$, while $1\leqslant_L^0 3\leqslant_L^0 2$, $1\leqslant_R^0 3\leqslant_R^0 2$, and we must have $1\leqslant_L^1 2$ and $1\leqslant_R^1 2$. This makes it impossible to complete the representation $R'$, which is a contradiction.
\end{proof}


\begin{thebibliography}{A00}

\bibitem{ADH} G. Agnarsson, P. Damaschke and M. M. Halld\'orsson. Powers of geometric intersection graphs and dispersion algorithms. \emph{Discrete Applied Mathematics} 132(1--3):3--16, 2003.

\bibitem{BW99} K. P. Bogart and D. B. West. A short proof that 'proper$=$unit'. \emph{Discrete Mathematics} 201:21--23.

\bibitem{DCP88} I. Dagan, M. C. Golumbic and R.Y. Pinter. Trapezoid graphs and their coloring. \emph{Discrete Applied Mathematics} 21:35--46, 1988.

\bibitem{F95} C. Flotow. On powers of $m$-trapezoid graphs. \emph{Discrete Applied Mathematics} 63(2):187--192, 1995.

\bibitem{F96} C. Flotow. On powers of circular arc graphs and proper circular arc graphs. \emph{Discrete Applied Mathematics} 69(3):199--207, 1996.

\bibitem{algopaper} F. Foucaud, G. B. Mertzios, R. Naserasr, A. Parreau and P. Valicov. Algorithms and complexity for metric dimension and location-domination on interval and permutation graphs. Proceedings of the 41st International Workshop on Graph-Theoretical Concepts in Computer Science (WG'15), Lecture Notes in Computer Science (to appear). Extended manuscript available at \url{http://arxiv.org/abs/1405.2424}



\bibitem{LS83} R. Laskar and D. Shier. On powers and centers of chordal graphs. \emph{Discrete Applied Mathematics} 6(2):139--147, 1983.

\bibitem{R87} A. Raychaudhuri. On powers of interval graphs and unit interval graphs. \emph{Congressus Numerantium} 59:235--242, 1987.

\bibitem{R69} F. S. Roberts. Indifference graphs. In \emph{Proof Techniques in Graph Theory}, F. Harary (Ed.), Academic Press, New York, pp. 139--146, 1969.

\end{thebibliography}
\end{document}